\documentclass[11pt,a4paper]{amsart}

\usepackage[a4paper]{geometry}
\usepackage[utf8]{inputenc}
\usepackage[T1]{fontenc}
\usepackage{lmodern}
\usepackage{amsmath}
\usepackage{amsfonts,amssymb,amsthm}
\usepackage{mathtools}
\usepackage[all]{xy}
\usepackage{hyperref}
\usepackage{verbatimbox}
\usepackage{tikz-cd}
\usepackage{geometry}
\usepackage{blkarray}
\usepackage{comment}
\usepackage{enumerate}
\usepackage{float}
\usepackage{mathrsfs}
\usepackage{relsize}
\usepackage{tabularx}
\usepackage{pdfpages}

\newtheorem{proposition}{Proposition}
\newtheorem{lemma}{Lemma}
\newtheorem{thm}{Theorem}

\theoremstyle{definition}

\newtheorem{rmk}{Remark}

\DeclareMathOperator{\ind}{Ind}

\DeclareMathOperator{\Lie}{Lie}

\DeclareMathOperator{\Hom}{Hom}

\DeclareMathOperator{\Ad}{Ad}
\DeclareMathOperator{\id}{Id}

\title{Filtrations of tilting modules and costalks of parity sheaves}

\author{Linyuan Liu}

\address{Sydney Mathematical Research Institute\\
University of Sydney\\
NSW 2006\\
Australia}

\email{Linyuan.Liu@normalesup.org}

\thanks{The work in this article was completed while the author is a Research
Associate at the University of Sydney supported by Australian Research
Council grant DP160103897. The author is grateful for the hospitality
of Sydney Mathematical Research Institute during her stay there. }



\date{25 February, 2020}
\begin{document}
\maketitle
\begin{abstract}
  In this article, we proved that the costalks of parity sheaves on
  the affine Grassmannian correspond to the Brylinski-Kostant
  filtration of the corresponding weight spaces of tilting modules.
\end{abstract}
\section{Introduction}
\subsection{Summary}
Assume \( G \) is a split reductive algebraic group over a field \( k
\). When \( k=\mathbb{C} \), R.K.Brylinski constructed a filtration of
weight spaces of a \( G \)-module, using the action of a principal
nilpotent element of the Lie algebra, and proved that this filtration
corresponds to Lusztig's q-analogue of the weight multiplicity
(cf.\cite{Bry89} ). Later, Ginzburg discovered that this filtration
has an interesting geometric interpretation via the geometric Satake
correspondence (cf. \cite{Gin89}). The goal of this article is to partially generalise
these results to the case where the characteristic of \( k \) is
positive.

\subsection{Main result}
In the rest of the article, let  \( G \) be a reductive group over \( k \) and is a product of
simply-connected quasi-simple groups and general linear groups.
Suppose \( k \) algebraically closed such that the characteristic is
 good for each quasi-simple factor of \( G \) in the sense of \cite{JMW}. Suppose there exists a non degenerate \( G \)-equivariant bilinear
form on \( \mathfrak{g} \). When there is no confusion, we write \(
\otimes \) for \(\otimes_{k} \). Fix a Borel subgroup \( B\subset G \)
and a maximal torus \( T\subset B \). Let \( \mathbf{X}=X^{*}(T) \) be
the weight lattice and \( \mathbf{X}^{+} \) be the set of dominant
weights with respect to \( B \).

Let \( \mathcal{G}r \) be the  affine Grassmannian variety of
the complex Langlands dual group \( \check{G} \) of \( G \). Let \(
\check{T} \subset \check{G} \) be the  maximal torus. For each \( \mu\in
\mathbf{X} \), let \( L_{\mu} \) be the corresponding \( \check{T}
\)-fixed point in \( \mathcal{G}r \), and let \( i_{\mu} \) be the
embedding \( \{L_{\mu}\}\hookrightarrow \mathcal{G}r \). When \( \mu \) is dominant,
denote by \( \mathcal{G}r^{\mu} \) the \(
\check{G}(\mathcal{O})=\check{G}(\mathbb{C}[[t]]) \)-orbit of \(
L_{\mu} \) in \( \mathcal{G}r \), by \( \mathcal{E}(\mu) \) the
indecomposable parity sheaf with respect to the stratum \( \mathcal{G}r^{\mu} \)
(cf. \cite{JMW}), and by \( \mathbf{T}(\mu) \) the indecomposable
tilting module of \( G \) of highest weight \( \mu \).

Denote by \( \mathfrak{g} \), \( \mathfrak{b} \) and \( \mathfrak{t} \) the Lie algebras
of \( G \), \( B \) and \( T \). The main result of this article is the following
\begin{thm}\label{thm:main}
Let \( e\in \mathfrak{b}\) be a principal nilpotent element that is \(
\mathfrak{t} \)-adapted (i.e., there exists \( h\in \mathfrak{t} \)
such that \( [h,e]=e \)). For all \( \lambda,\mu\in\mathbf{X}^{+} \), let \(
\mathbf{F}_{\bullet}(\mathbf{T}(\lambda)_{\mu}) \) be the Brylinski-Kostant
filtration of \( \mathbf{T}(\lambda)_{\mu} \) defined by \( e \), i.e. 
  for all \( n\in\mathbb{N} \), we have
  \begin{displaymath}
    \mathbf{F}_{n}(\mathbf{T}(\lambda)_{\mu})=\{v\in \mathbf{T}(\lambda)_{\mu}\mid
    e^{(i+1)}v=0\text{ for all }i\geq n\},
  \end{displaymath}
and \(  \mathbf{F}_{n}(\mathbf{T}(\lambda)_{\mu})=0 \) whenever \( n<0 \).

Then we have
  \begin{equation}
    \label{eq:7786ed48df651f44}
    \dim
    \mathbf{H}^{2n-\dim(\mathcal{G}r^{\mu})}(i_{\mu}^{!}\mathcal{E}(\lambda))=\dim \big(\mathbf{F}_{n}(\mathbf{T}(\lambda)_{\mu})/\mathbf{F}_{n-1}(\mathbf{T}(\lambda)_{\mu})\big)
  \end{equation}
for all \( n\in \mathbb{Z} \).
\end{thm}

The proof in this article follows exactly the same idea with
\cite{Bry89} and
Proposition 2.6.3 in \cite{GR15}.

The author thanks Simon Riche and Geordie Williamson for many
useful discussions.

\section{Proof of the main result}
Let \( \mathfrak{n}=\Lie(U) \), where \( U \) is the unipotent radical
of \( B \). Let \(
\tilde{\mathfrak{g}}=G\times^{B}(\mathfrak{g}/\mathfrak{n})^{*} \) be
the Grothendieck resolution. Then we have an isomorphism of graded \( k[\mathfrak{t}^{*}]\cong H^{\bullet}_{\check{T}}(\text{pt}) \)-modules
(cf. \cite{MR15} Prop.1.9)
\begin{equation}
  \label{eq:b3126ad8387fcdf4}
  \mathbf{H}_{\check{T}}^{\bullet-\dim(\mathcal{G}r^{\mu})}(i_{\mu}^{!}\mathcal{E}(\lambda))\cong
  (\mathbf{T}(\lambda)\otimes \Gamma(\tilde{\mathfrak{g}},\mathcal{O}_{\tilde{\mathfrak{g}}}(-w_{0}\mu)))^{G}
\end{equation}
where on the right hand side \( \mathbf{T}(\lambda) \) is in degree \( 0 \) and
the global sections are equipped with the grading induced by the \(
\mathbb{G}_{m} \)-action on \( \tilde{\mathfrak{g}} \) defined by
\begin{displaymath}
  z\cdot(g\times^{B}x)=g\times^{B}(z^{2}x).
\end{displaymath}

\begin{lemma}
  Let \( P\subset G \) be a parabolic subgroup such that \( G\to G/P
  \) is locally trivial and let \( V,V' \) be
  \( P \)-modules. Let \( \mathbb{G}_{m} \) act on \( V \) by \(
  z\cdot x=z^{2}x \). Let \( \pi:G\times^{P}V \to G/P\) be the natural
  map. Then we have an isomorphism of graded \( G
  \)-modules
  \begin{equation}
    \label{eq:f9ec0f3bb5480d84}
    \Gamma(G\times^{P}V,\pi^{*}\mathcal{L}_{G/P}(V'))=\ind_{P}^{G}(V'\otimes
    k[V])
  \end{equation}
  where \( \mathcal{L}_{G/P}(V') \) is the associated sheaf 
  induced by the \( P \)-module \( V' \) in the sense of \cite{Jan03}
  I.5.8 with \( X=G\times V \),  the
  grading on the left hand side is induced by the action of \(
  \mathbb{G}_{m} \) on \( V \), and the grading on the right hand side
  is induced by the grading on \( k[V]\cong S(V^{*}) \) with \( V^{*}
  \) placed on degree \( 2 \).
\end{lemma}
\begin{proof}
Let \( X=G\times V \), then \( P \) acts on \( X \) by \( (g,x)\cdot
p=(gp,p^{-1}x) \) and we have \( G\times^{P}V=X/P \) by
definition. Then we have
\begin{displaymath}
  \pi^{*}(\mathcal{L}_{G/P}(V'))\cong \mathcal{L}_{X/P}(V')
\end{displaymath}
by \cite{Jan03} I.5.17 (1). Now we have
\begin{align*}
  \Gamma(G\times^{P}V,\pi^{*}\mathcal{L}_{G/P}(V'))&=\Gamma(X/P,
                                                     \mathcal{L}_{X/P}(V'))\\
                                                   &=(V'\otimes k[X])^{P}\\
                                                   &=(V'\otimes k[G]\otimes k[V])^{P}\\
                                                   &=(k[G]\otimes (V'\otimes k[V]))^{P}\\
  &=\ind_{P}^{G}(V'\otimes k[V])
\end{align*}
with the desired gradings.
\end{proof}

Apply the lemma to \( P=B \),\( V=(\mathfrak{g}/\mathfrak{n})^{*} \) and \( V'=k_{-w_{0}\mu}
\), we get an isomorphisme of graded \( G \)-modules
\begin{displaymath}
  \Gamma(\tilde{\mathfrak{g}},\mathcal{O}_{\tilde{\mathfrak{g}}}(-w_{0}\mu))\cong
  \ind_{B}^{G}(k_{-w_{0}\mu}\otimes k[(\mathfrak{g}/\mathfrak{n})^{*}]).
\end{displaymath}
Hence we have an isomorphism of graded \( k[\mathfrak{t}^{*}] \)-modules
\begin{displaymath}
  \mathbf{H}_{\check{T}}^{\bullet-\dim(\mathcal{G}r^{\mu})}(i_{\mu}^{!}\mathcal{E}(\lambda))\cong
  (\mathbf{T}(\lambda)\otimes \ind_{B}^{G}(k_{-w_{0}\mu}\otimes k[(\mathfrak{g}/\mathfrak{n})^{*}]))^{G}\cong
  (\mathbf{T}(\lambda)\otimes k_{-w_{0}\mu}\otimes k[(\mathfrak{g}/\mathfrak{n})^{*}])^{B}
\end{displaymath}
by tensor identity and Frobenius reciprocity.

Take a regular semisimple element \( \phi\in \mathfrak{t}^{*} \). Then we have
isomorphisms of filtered vector spaces
\begin{equation}
  \label{eq:57e826ff89859876}
  \mathbf{H}^{\bullet-\dim(\mathcal{G}r^{\mu})}_{\phi}(i_{\mu}^{!}\mathcal{E}(\lambda))\cong (\mathbf{T}(\lambda)\otimes k_{-w_{0}\mu}\otimes
   k[(\mathfrak{g}/\mathfrak{n})^{*}])^{B}\otimes_{k[\mathfrak{t}^{*}]}k_{\phi}.
 \end{equation}

 Identify \( \mathfrak{g} \) with \( \mathfrak{g}^{*} \) by a non degenerate \( G
 \)-equivariant bilinear form and let \( h\in \mathfrak{t} \) be the
 image of \( \phi \), which is a regular semisimple. Then we have 
 \begin{equation}
   \label{eq:e5fbd3818496a225}
     \mathbf{H}_{\check{T}}^{\bullet-\dim(\mathcal{G}r^{\mu})}(i_{\mu}^{!}\mathcal{E}(\lambda))\cong 
     (\mathbf{T}(\lambda)\otimes k_{-w_{0}\mu}\otimes
     k[\mathfrak{b}])^{B}\otimes_{k[\mathfrak{t}]}k_{h}
   \end{equation}

 To transform the formula above to a form with which is easier to
 deal, we also need the following three lemmas. 

 \begin{lemma}\label{lemma:Jantzen}
   If \( h\in\mathfrak{t}_{\text{rs}} \), then we have a \( B \)-equivariant isomorphism
   of varieties
   \begin{displaymath}
     (h+\mathfrak{n})\times \mathfrak{t}_{\text{rs}}\xrightarrow{\sim}
     \mathfrak{b}\times_{\mathfrak{t}}\mathfrak{t}_{\text{rs}}
   \end{displaymath}
   such that \( (h,h)\mapsto h\times_{\mathfrak{t}}h \) and the following diagramme 
   \begin{displaymath}
     \begin{tikzcd}[column sep=tiny]
     (h+\mathfrak{n})\times
     \mathfrak{t}_{\text{rs}}\ar[rr,"\sim"]\ar[rd,"p_{2}"]&&\mathfrak{b}\times_{\mathfrak{t}}\mathfrak{t}_{\text{rs}}\ar[ld,"\pi_{2}"]\\
     &\mathfrak{t}_{\text{rs}}&
     \end{tikzcd}
   \end{displaymath}
   is an isomorphism of affine bundles over \( \mathfrak{t}_{\text{rs}} \).
 \end{lemma}
 \begin{proof}
   First, let us construct a map \( \Phi: (h+\mathfrak{n})\times \mathfrak{t}_{\text{rs}}\rightarrow
     \mathfrak{b}\times_{\mathfrak{t}}\mathfrak{t}_{\text{rs}} \). Let
     \( (X,H)\in  (h+\mathfrak{n})\times \mathfrak{t}_{\text{rs}} \),
     then there exists \( b\in B \) such that \( \Ad(b)(h)=X \). We set
     \( \Phi(X,H)=(\Ad(b)(H),H) \), which indeed lies in \(
     \mathfrak{b}\times_{\mathfrak{t}}\mathfrak{t}_{\text{rs}}  \)
     since by the projection \(\mathfrak{b}\to\mathfrak{t}  \), the
     image of \( \Ad(b)(H) \) is \( H\in \mathfrak{t}_{\text{rs}}
     \). We need to check
     \begin{itemize}
     \item \( \Phi \) is well-defined, which means it doesn't depend
       on the choice of \( b\in B \);
     \item \( \Phi \) is \( B \)-equivariant (obvious);
     \item \( \Phi \) is a morphism of varieties;
     \item \( \Phi \) is bijective;
     \item \( \Phi^{-1} \) is a morphism of varieties;
     \end{itemize}
  If \( X=\Ad(b)(h)=\Ad(b')(h) \), then \( Ad(b^{-1}b')(h)=h \), hence
  \( b^{-1}b'\in T \) since \( h \) is regular semisimple in \(
  \mathfrak{t} \), and \( \Ad(b)(H)=\Ad(b')(H) \) since \(
  H\in\mathfrak{t} \). To prove that \( \Phi \) is a morphism, observe
  that \( \Phi \) is induced by the following commutative diagramme
  \begin{displaymath}
    \begin{tikzcd}
      (h+\mathfrak{n})\times
      \mathfrak{t}_{\text{rs}}\ar[d,"p_{2}"]\ar[r,"\phi\times \id"]& U\times
      \mathfrak{t}_{\text{rs}}\ar[r,"\psi"]&\mathfrak{b}\ar[d]\\
      \mathfrak{t}_{\text{rs}}\ar[rr]&&\mathfrak{t}
    \end{tikzcd}
  \end{displaymath}
  where \( \psi(u,H)=\Ad(u)(H) \), and \( \phi \) is the inverse map of
  \( U\to h+\mathfrak{n}: u\mapsto \Ad(u)(h) \) which is a morphism
  by \cite{Jan04} page 188. Bijectivity is easy to prove. \( \Phi^{-1}
  \) is also a morphism because it is the composition
  \begin{displaymath}
    \mathfrak{b}\times_{\mathfrak{t}}\mathfrak{t}_{\text{rs}}
    \xrightarrow{\pi_{1}}\mathfrak{b}_{\text{rs}}\xrightarrow{f}U\times
    \mathfrak{t}_{\text{rs}}\xrightarrow{g} (h+\mathfrak{n})\times
      \mathfrak{t}_{\text{rs}}
    \end{displaymath}
    where \( f \) is the map on page 188 \cite{Jan04} and \(
    g(u,H)=(\Ad(u)(h),H) \), which are both morphisms.
 \end{proof}

\begin{lemma}\label{lemma:tensor}
  Let \( H \) be an algebraic group over \( k \), \( A \) a flat \(
  k \)-algebra and \( M \) an \( H \)-module. Then the
  natural map \( M^{H}\otimes_{k} A \to (M\otimes_{k}A)^{H}\) is an
  isomorphism of \( A \) modules, where the action of \( H \) on \(
  M\otimes_{k}A \) is induced by the action of \( H \) on \( M \) and
  the trivial action of \( H \) on \( A \).
\end{lemma}
\begin{rmk}
  The assumption of flatness of \( A \) is automatically satisfied,
  since \( k \) is a field. We include this assumption since the
  statement is also correct even in more general cases, where the
  flatness will be crucial.
\end{rmk}
\begin{proof}
  Since \( A \) is flat, it is easy to check that the map 
  \( M^{H}\otimes_{k} A \to (M\otimes_{k}A)^{H}, \; m\otimes a\mapsto
  m\otimes a \) is a bijection.
\end{proof}

\begin{lemma}\label{lemma:localisation}
  Let \( H \) be an algebraic group over \( k \) and \( A \) a flat \(
  k \)-algebra. Let \( M \) be an \( H \)-module and a torsion free \( A \)-module such that the
  two actions commute (i.e. \( h\cdot (am)=a(h\cdot m) \) for all
  \( m\in M \), \( a\in A \) and \( h\in H \)). Then for any
  multiplicative subset \( S\subset A \), the natural morphism
  \begin{displaymath}
    S^{-1}(M^{H})\to (S^{-1}M)^{H}
  \end{displaymath}
  is an isomorphism of \( S^{-1}A \)-modules.
\end{lemma}

\begin{proof}
  The map is induced by \( s^{-1}m\mapsto s^{-1}m \).
\end{proof}


 Using the above lemmas, we have isomorphisms of filtered vector spaces
 \begin{align*}
 & (\mathbf{T}(\lambda)\otimes k_{-w_{0}\mu}\otimes
   k[\mathfrak{b}])^{B}\otimes_{k[\mathfrak{t}]}k_{h}\\
 =&(\mathbf{T}(\lambda)\otimes k_{-w_{0}\mu}\otimes
   k[\mathfrak{b}])^{B}\otimes_{k[\mathfrak{t}]}k[\mathfrak{t}_{\text{rs}}]\otimes_{k[\mathfrak{t}_{\text{rs}}]}
    k_{h}  \\
  =&(\mathbf{T}(\lambda)\otimes k_{-w_{0}\mu}\otimes
   k[\mathfrak{b}]\otimes_{k[\mathfrak{t}]}k[\mathfrak{t}_{\text{rs}}])^{B}\otimes_{k[\mathfrak{t}_{\text{rs}}]}
    k_{h}  \\
 =&(\mathbf{T}(\lambda)\otimes k_{-w_{0}\mu}\otimes
   k[h+\mathfrak{n}]\otimes k[\mathfrak{t}_{\text{rs}}])^{B}\otimes_{k[\mathfrak{t}_{\text{rs}}]}
    k_{h}  \\
 =&(\mathbf{T}(\lambda)\otimes k_{-w_{0}\mu}\otimes
   k[h+\mathfrak{n}])^{B}\otimes k[\mathfrak{t}_{\text{rs}}]\otimes_{k[\mathfrak{t}_{\text{rs}}]}
    k_{h}  \\
=& (\mathbf{T}(\lambda)\otimes k_{-w_{0}\mu}\otimes
   k[h+\mathfrak{n}])^{B},
 \end{align*}
 where the second isomorphism is due to Lemma
 \ref{lemma:localisation}, the third is due to Lemma
 \ref{lemma:Jantzen}, and the fourth is due to Lemma
 \ref{lemma:tensor}.

 Hence there is an isomorphism of filtered vector spaces
 \begin{displaymath}
   \mathbf{H}^{\bullet-\dim(\mathcal{G}r^{\mu})}_{\phi}(i_{\mu}^{!}\mathcal{E}(\lambda))\cong (\mathbf{T}(\lambda)\otimes k_{-w_{0}\mu}\otimes
   k[h+\mathfrak{n}])^{B}.
 \end{displaymath}

 On the other hand, by using the geometric Satake equivalence and
 equivariant localisation,  the left-hand side is isomorphic to the vector
 space \( \mathbf{T}(\lambda)_{\mu} \), hence in particular we have
 \begin{displaymath}
   \dim \mathbf{T}(\lambda)_{\mu}=\dim (\mathbf{T}(\lambda)\otimes k_{-w_{0}\mu}\otimes
   k[h+\mathfrak{n}])^{B}.
 \end{displaymath}

 \begin{lemma}
   Let \( M \) be a \( B \)-module and \( \mu\in X(T) \). Then there
   exists a natural isomorphism
   \begin{displaymath}
     (M\otimes k_{-\mu})^{B}\cong (M^{U})_{\mu}
   \end{displaymath}
   defined by sending \( m\otimes 1 \) to \( m \).
 \end{lemma}
 \begin{proof}
   \begin{displaymath}
      (M\otimes k_{-\mu})^{B}\cong \Hom_{B}(k_{\mu},M)\cong (M^{U})_{\mu}.
   \end{displaymath}
 \end{proof}
 \begin{lemma}
   The map
   \begin{displaymath}
     \Lambda: (\mathbf{T}(\lambda)\otimes k[h+\mathfrak{n}])^{U}\rightarrow \mathbf{T}(\lambda)
   \end{displaymath}
   defined by evaluation on \( h \) is an isomorphism of \( T
   \)-modules.

   In particular, it induces an isomorphism of vector spaces:
    \begin{displaymath}
   \Lambda_{\mu}: (\mathbf{T}(\lambda)\otimes k_{-\mu}\otimes
   k[h+\mathfrak{n}])^{B} \cong \mathbf{T}(\lambda)_{\mu}.
 \end{displaymath}
\end{lemma}
\begin{proof}
  \( \Lambda\) is \( T \)-equivariant because \( h \) is fixed by \( T \).

  On the other hand, we already have  
   \begin{displaymath}
\dim \mathbf{T}(\lambda)_{\mu}=\dim (\mathbf{T}(\lambda)\otimes k_{-w_{0}\mu}\otimes
   k[h+\mathfrak{n}])^{B}=\dim ( (\mathbf{T}(\lambda)\otimes k[h+\mathfrak{n}])^{U})_{\mu}     
 \end{displaymath}
 because the dimension of the weight spaces with respect to \( \mu \)
 and \( w_{0}\mu \) are the same.
 By taking the sum over all \( \mu \), we have \( \dim
 (\mathbf{T}(\lambda)\otimes k[h+\mathfrak{n}])^{U} =\dim \mathbf{T}(\lambda)\). Hence
 it suffices to prove that \( \Lambda \) is injective because both
 sides are finite dimensional.

 The idea of the proof of injectivity is quite simple. Roughly
 speaking, 
   an \( U \)-equivariant function on \( h+\mathfrak{n} \) is zero if
    it is zero on \( h \). The following is just a more rigorous version of
    this simple idea.

 Identify \((\mathbf{T}(\lambda)\otimes
 k[h+\mathfrak{n}])^{U}=\Hom_{U}(\mathbf{T}(\lambda)^{*}, k[h+\mathfrak{n}])
 \) and \( \mathbf{T}(\lambda)=(\mathbf{T}(\lambda)^{*})^{*} \). Then for \( f\in
 \Hom_{U}(\mathbf{T}(\lambda)^{*}, k[h+\mathfrak{n}])  \), \( \Lambda(f):
 \mathbf{T}(\lambda)^{*}\to k \) is defined by \( \Lambda(f)(\psi)=f(\psi)(h)
 \). Hence if \( \Lambda(f)=0 \), then for all \( \psi\in
 \mathbf{T}(\lambda)^{*} \) and \( u\in U \), we have \(
 f(\psi)(\Ad(u)(h))=f(u^{-1}\psi)(h)=\Lambda(f)(u^{-1}\psi)=0 \).
 But since \( h \) is principal semi-simple, we have \(
 \Ad(U)(h)=h+\mathfrak{n} \), hence \( f(\psi)(X)=0 \) for all \( X\in
 h+\mathfrak{n} \). Since \( k \) is an infinite field, this means
 that as an element in \( k[h+\mathfrak{n}] \), we have \( f(\psi)=0
 \) (another way to think about this: \( k \) is algebraically closed
 and \(k[h+\mathfrak{n}]  \) is reduced, then if some function is zero
 at each closed point, it is zero by Hilbert's
 Nullstellensatz.). Since \( \psi \) is arbitrary, we have \( f=0
 \). This proves the injectivity.
\end{proof}

We conclude the proof of Theorem \ref{thm:main} by the following
\begin{proposition}
  Let \( e\in \mathfrak{n} \) such that \( [h,e]=e \). Then \( e \) is
  a principal nilpotent.  
  Then we have
  \begin{displaymath}
    f\in \Hom_{B}(\mathbf{T}(\lambda)^{*}\otimes k_{\mu},
    k[h+\mathfrak{n}]_{n})\Leftrightarrow \Lambda(f)\in \mathbf{F}_{n}(\mathbf{T}(\lambda)_{\mu})
  \end{displaymath}
  for all \( n\in\mathbb{N} \).
\end{proposition}

 \begin{rmk} Roughly speaking, the idea of the proof is as follows: 
        if a \( B \)-equivariant map from \( \mathbf{T}(\lambda)^{*}\otimes
        k_{\mu} \) to 
   \( k[h+\mathfrak{n}] \) takes any element to a polynomial that has
   degree \( \leq n \) along the direction \( e\in \mathfrak{n} \),
   then it takes any element to a polynomial with degree \( \leq n \),
   because \( B\cdot e \) is dense.
   We will make this idea rigorous in the proof.
      \end{rmk}

\begin{proof}
  Denote \(V=\mathbf{T}(\lambda)  \)
  Fix \( f\in \Hom_{B}(V^{*}\otimes k_{\mu},
    k[h+\mathfrak{n}]) \) and let \( v=\Lambda(f) \). Then \( 
    f\in \Hom_{B}(V^{*}\otimes k_{\mu},
    k[h+\mathfrak{n}]_{n}) \) if and only if for any \( \psi\in V^{*}
    \),  \( f(\psi\otimes 1)\in k[h+\mathfrak{n}] \) has degree \(
    \leq n \).
     Since \( k \) is an infinite field, \( f(\psi\otimes 1)\in k[h+\mathfrak{n}] \) has degree \(
    \leq n \) if and only if for all \( X\in \mathfrak{n} \), the
    polynomial in \( t \)
    \begin{displaymath}
      f(\psi\otimes 1)(h+tX )
    \end{displaymath}
    has degree \( \leq n \).
    
     Since \(B\cdot e \) is dense in \(
    \mathfrak{n} \), we have \( f\in \Hom_{B}(V^{*}\otimes k_{\mu},
    k[h+\mathfrak{n}]_{n}) \) if and only if it satisfies the
    following condition (A):

    ``For all \( \psi\in V^{*}
    \) and all \( b\in B \),
    the polynomial (in
    \( t \))
      \( f(\psi\otimes 1)(h+t b\cdot e) \) 
    has degree \( \leq n \).''

    Claim: (A) is equivalent to the condition (B):

    ``For all \( \psi\in V^{*}
    \),
    the polynomial (in
    \( t \))
      \( f(\psi\otimes 1)(h+te) \)
      has degree \( \leq n \).''

    Proof of the claim: (A) clearly implies (B). Now suppose \( f \)
    satisfies (B). Fix \( \psi\in V^{*} \) arbitrary, choose a \(
    b_{0}\in B \) such that the polynomial \( f(\psi\otimes 1)\in
    k[h+\mathfrak{n}] \) reaches maximal degree in the direction \(
    b_{0}\cdot e\in \mathfrak{n} \) (such a \( b_{0} \) exists because
    \( k \) is infinite and \( B\cdot e \) is dense in \( \mathfrak{n}
    \)). Then for \( b\in B \) arbitrary, the degree of \(
    f(\psi\otimes 1)(h+t b\cdot e) \) is no larger than that of \(
    f(\psi\otimes 1)(h+t b_{0}\cdot e) \). But since the latter is
    maximal,  it is the same with the degree of 
    \begin{displaymath}
     f(\psi\otimes 1)(b_{0}h+t b_{0}\cdot e)=f(b_{0}^{-1}(\psi\otimes
     1))(h+te)=f((\mu(b_{0})^{-1}b_{0}^{-1}\cdot\psi)\otimes 1)(h+te),
   \end{displaymath}
   which is \( \leq n \) by applying (B) to
   \(\mu(b_{0})^{-1}b_{0}^{-1}\cdot\psi\in V^{*}  \). This finishes
   the proof of the claim.

   Using
    \begin{displaymath}
    f(\psi\otimes 1)(h+te)=  f(\psi\otimes
    1)(\exp(te)\cdot h)=f(\exp(-te)(\psi\otimes
    1))(h). \end{displaymath} and the claim, we have 
   \( 
    f\in \Hom_{B}(V^{*}\otimes k_{\mu},
    k[h+\mathfrak{n}]_{n}) \) if and only if for any \( \psi\in V^{*}
    \), the polynomial in \( t \)
    \begin{displaymath}
     f(\exp(-te)(\psi\otimes
    1))(h) 
   \end{displaymath}
   has degree \( \leq n \). But the element in \( (V^{*})^{*} \)
   sending \( \psi\in V^{*} \) to \( f(\exp(-te)(\psi\otimes
    1))(h)=f((\exp(-te)\psi)\otimes
    1)(h)
   \) is just \(\exp(te) \Lambda(f)=\exp(te)\cdot v \). 
\end{proof}
\def\refname{References}

\end{document}